\documentclass[11pt,a4paper]{article}
\usepackage{amsmath, amssymb, amsthm}
\usepackage{geometry}
\usepackage{xcolor}
\usepackage{hyperref}
\newtheorem{theorem}{Theorem}[section]
\newtheorem{lemma}[theorem]{Lemma}
\newtheorem{example}[theorem]{Example}

\DeclareMathOperator{\lcm}{lcm}
\DeclareMathOperator{\osc}{osc}
\newcommand{\e}{\exp}

\title{A Rational-Level  Criterion on Box Dimension of the Graph of Generalized Riemann-Type Functions}
\author{\sl Yurong Wu$^{a}$ and Guoping Zhan$^{b*}$}

\date{{\small \em $^a$School of Mathematical Sciences, Zhejiang University of Technology,\\
Hangzhou, 310023, P. R. China \\
e-mail: wuyurong2022@zjut.edu.cn; wuyurong2003@163.com\\
\ $^b$\thanks{Corresponding author: Guoping Zhan (e-mail: zhangp@zjut.edu.cn)}
School of Mathematical Sciences, Zhejiang University of Technology,\\
Hangzhou, 310023, P. R. China
\\e-mail: zhangp@zjut.edu.cn}}

\begin{document}
\maketitle

\begin{abstract}
We consider the box dimension of the graphs of the generalized Riemann-type functions
$G_\delta(x)=\sum_{n=1}^{\infty}g(n^{2}x)n^{-1-\delta}$ with 1-periodic real-valued continuous functions $g$
and $0<\delta\le 1$. Firstly, we establish  a rational-level non-vanishing criterion for the lower bound of lower box dimension of
the graph of $G_\delta$.
More precisely, We prove that the lower bound $\dim_B(\mathrm{graph}\,G_\delta)\ge\frac74-\frac\delta2$ under a mild decay condition of
the Fourier coefficients of $g$ and non-vanishing of the square-class chirp functional
$S_{d}(a;q)$ at a single rational $a/q$.
 A resolution theorem then asserts that for any nonconstant
real trigonometric polynomial $g$, the chirp functional $S(a;q)$ cannot vanish at every rational simultaneously;
consequently, $\dim_B(\mathrm{graph}\,G_\delta)=\frac74-\frac\delta2$ for all such $g$ with $0<\delta\le1$
which gives a negative answer to \cite[problem 2]{Wu-Zhan2026}.
Finally, two guiding examples distinguish structural vanishing
from genuinely arithmetic vanishing related to  modular elliptic curve and governed by the Prime Number Theorem.
\end{abstract}


\tableofcontents

\section{Introduction}

\subsection{Background}
The celebrated Weierstrass function
\[W(x)=\sum_{n=0}^{\infty}a^n\cos(b^n\pi x),\]
with odd integer $b$, $a\in(0,1)$ and $ab>1+3\pi/2$ is continuous but nowhere differentiable function,
which was proved by Weierstrass in 1872. More generally, Hardy \cite{Hardy1916} later showed that
for $a\in(0,1)$ and $b>1$ with $ab\ge1$, $W$ is nowhere differentiable.

\par The fractal geometry of the graphs of Weierstrass-type functions has since attracted considerable attention.
For instance, for functions of the form $W(x)=\sum_{i=0}^{\infty}\lambda^{-\alpha i}g(\lambda^ix)$,
where $\lambda>1$, $0<\alpha<1$ and $g$ is an almost periodic Lipschitz function of order greater
than $\alpha$, Hu and Lau \cite{Hu-Lau1993} introduced a novel fractal measure $\mathcal{K}_s$ ($1\le s\le2$) and
showed that the $\mathcal{K}$-dimension of the graph equals to $2-\alpha$.  Hunt \cite{Hunt1998} proved that for
Weierstrass-type functions $w_\Theta(x)=\sum_{n=0}^\infty a^n\cos(2\pi(b^n x+\theta_n))$ with random phases, where $\theta_n\in\Theta=\{\theta_0,\theta_1,...\}$ are independent and uniformly distributed on $[0,1]$, the Hausdorff
dimension of the graph is almost surely $D=2+\frac{\log a}{\log b}$. Heurteaux \cite{Heurteaux2003} established
the corresponding result for the box dimension of its graph. Bara\'nski, B\'ar\'any and Romanowska \cite{Baranski-Barany-Romanowska2014}
employed smooth ergodic theory (Pesin theory) combined with transversality to study the classical Weierstrass function
$W_{\lambda,b}(x)=\sum_{n=0}^{\infty} \lambda^n \cos(2\pi b^n x)$ with integer $b\ge2$ and $1/b<\lambda<1$. They proved that
for each $b$, there exists $\lambda_b\in(1/b,1)$ such that for every $\lambda\in(\lambda_b,1)$, the Hausdorff dimension of
its grap is $D=2+\log\lambda/\log b$.

\par Recent years have witnessed major breakthroughs in this area. Shen \cite{Shen2018} resolved a longstanding conjecture
by proving that for every integer $b\ge2$ and every $\lambda\in(1/b,1)$, the graph of the classical Weierstrass function
$W_{\lambda,b}(x)=\sum_{n=0}^{\infty}\lambda^n\cos(2\pi b^nx)$ has Hausdorff dimension $D=2+\frac{\log\lambda}{\log b}$.
Subsequently, Ren and Shen \cite{Ren-Shen2021} established a dichotomy theorem for Weierstrass-type functions $W(x)=\sum_{n=0}^{\infty}\lambda^n\phi(b^nx)$, which states that for a real-analytic periodic seed function $\phi$,
integer $b\ge2$, and $1/b<\lambda<1$, either $W$ is real-analytic, or its graph has Hausdorff dimension $D=2+\log_b\lambda$.
This generalizes the result of classical Weierstrass function. In a further development, Ren and Shen \cite{Ren-Shen2024} considered the high-dimensional analogue $W(x)=\sum_{n=0}^{\infty}\lambda^n\phi(b^nx)$, where $\phi:\mathbb{R}\to\mathbb{R}^d$ is real-analytic and periodic,
$b\ge2$ is an integer, and $\lambda\in(1/b,1)$. They proved that both the box dimension and the Hausdorff dimension of its graph are given by
$\min\left\{\log_{\lambda^{-1}}b,\;1+(d-q)\left(1+\log_b\lambda\right)\right\}$,
where $q=q(\phi,b,\lambda)$ is the maximal dimension of all linear subspaces $V<\mathbb{R}^d$ for which the projection
$\pi_V W$ is Lipschitz. This fully extends Shen's breakthrough \cite{Shen2018} to higher dimensions.

Additionally, Buczolich, K\"aenm\"aki and Maga \cite{Buczolich-Kaenmaki-Maga2025} investigated the level sets of the
$\alpha$-Weierstrass function $W_{g}^{\alpha,b}(x)=\sum_{k=0}^{\infty}b^{-\alpha k}g(b^kx)$, where $g$ is a Lipschitz function
on the circle. They showed that for a prevalent
$\alpha$-Weierstrass function, the upper box dimension of every level set is at most $1-\alpha$, while the Hausdorff dimension
of almost every level set (with respect to its occupation measure) is exactly $1-\alpha$.

\par In parallel, Riemann introduced in the 19th century the renowned  function \[R(x)=\sum_{n=1}^{\infty}n^{-2}\sin(\pi n^2x)\]
and conjectured that it is everywhere non-differentiable, though he did not supply a rigorous proof. It was not until 1916 that
Hardy \cite{Hardy1916}  proved that $R$ fails to be differentiable at every irrational point and at certain rational points.
Later, Gerver \cite{Gerver1970} showed that $R$ is in fact differentiable at rational points of the form $(2p+1)/(2q+1)$
with $p, q\in\mathbb{Z}$.

\par After that,  in 1991 Duistermaat \cite{Duistermaat1991} conducted a systematic and deep study of the local structure
of the Riemann function near rational points, revealing a rich self-similar structure in its graph. And then,
Jaffard \cite{Jaffard1996} determined the H\"older regularity of $R$ at every point and computed its full singularity spectrum,
thereby establishing that the Riemann function is inherently multifractal. Also, Chamizo and Ubis \cite{Chamizo-Ubis2014}
studied polynomial Fourier series $F(x)=\sum_{n=1}^{\infty}n^{-\alpha}e^{2\pi iP(n)x}$, where $P$ is a polynomial of degree $k>1$.
For $1+k/2<\alpha<k$, they obtained upper and lower bounds for the singularity spectrum of $F$, concluding that $F$ is also multifractal.
Meanwhile, Chamizo et al. \cite{Chamizo-Petrykiewicz-Ruiz2017} studied the local regularity of fractional integrals of Fourier series, with particular emphasis on series arising from fractional integrals of modular forms. Their results indicate that, in general,
cusp forms give rise to pure fractals rather than multifractals.

Regarding the fractal dimension of Riemann-type functions, Chamizo and C\'{o}rdoba \cite{Chamizo-Cordoba1999} proved that for
$F_{\alpha,k}=\sum_{n=1}^\infty c_ne^{2\pi in^kx}/n^{\alpha}$, where $0<\underline{\lim}\,c_n\le\overline{\lim}\,c_n<\infty$
and $(k+1)/2\le\alpha\le k+1/2$, the box dimension of the graph satisfies
$\dim_{\mathrm{B}}\bigl(graph(F_{\alpha,k})\bigr)=2+1/2k-\alpha/k$. They further provided a complete characterisation
of the differentiability $f_k(x)=\sum_{n\geq1}n^{-k}e^{2\pi in^kx}$ at rational points. In the special case $c_n=1$.
Chamizo and Ubis  studied Chamizo and Ubis \cite{Chamizo-Ubis2007} studied the box dimension of the graph of
$f_{\alpha,k}(x)=\sum_{n=1}^\infty n^{-\alpha}e^{2\pi in^kx}$ and proved that
$\dim_{\mathrm{B}}\bigl(graph(f_{\alpha,k})\bigr)=\max(1,\,2+1/2k-\alpha/k)$ for $\alpha>(k+2)/4$,
thereby extending the earlier box-dimension formula for $F_{\alpha,k}$ of the constant-coefficient case.
Moreover, Chamizo \cite{Chamizo2004} considered Fourier series defined via a fractional integral of automorphic
forms and investigated their differentiability properties, as well as the fractal dimension of the graphs of their
real and imaginary parts. Later, C\'ordoba \cite{Cordoba2008} established that for
$F_{\delta}(x)=\sum_{n=1}^{\infty}\sin(2\pi n^{2}x)/n^{1+\delta}$ with $0<\delta\le1$, which corresponds $g=sin$,
the box dimension of the graph is $\dim_{\mathrm{B}}\bigl(graph(F_{\delta})\bigr)=\frac{7}{4}-\frac{\delta}{2}$.
The proof relies on quadratic residues for the lower bound and on Weyl-type incomplete Gauss sums with Farey
dissection for the upper bound; the cosine case is analogous.

\par The geometric properties of complex analogues of the Riemann function have also attracted considerable interest.
Eceizabarrena \cite{Eceizabarrena2020,Eceizabarrena2021} studied the complex Riemann function
$\phi(t)=\sum_{k\in\mathbb{Z}}(e^{-4\pi^2ik^2t}-1)/(-4\pi^2k^2)$, which arises numerically as the
limiting trajectory of a corner of a regular polygonal vortex filament under the binormal flow as the number
of sides tends to infinity.  In \cite{Eceizabarrena2020}, the author proved that $\phi(\mathbb{R})$ admits no
tangent anywhere in the sense of Falconer's geometric tangent definition. In \cite{Eceizabarrena2021}, he showed that the image
$\phi(\mathbb{R})$ has Hausdorff dimension satisfying $1\le\dim_{\mathcal{H}}\phi(\mathbb{R})\le4/3$, and that for the set
$D_{\alpha}$ of points with H\"older exponent $\alpha\in[1/2,\,3/4]$, one has $\dim_{\mathcal{H}}\phi(D_{\alpha})\le(4\alpha-2)/\alpha$.

\par Furthermore, Banica and Vega \cite{Banica-Vega2022} proved the existence of solutions to the binormal flow
with smooth trajectories that can be made arbitrarily close to curves exhibiting multifractal behaviour,
thus establishing a direct link between this classical analytic object and a nonlinear geometric PDE.
Banica et al.\cite{Banica-Eceizabarrena-Nahmod-Vega2025} combined restriction-type Diophantine approximation
from analytic number theory with multifractal analysis from harmonic analysis to give a complete description of
the singularity spectrum of the generalised Riemann function $R_{x_0}(t)=\sum_{n\neq0}e^{2\pi i(n^2t+nx_0)}/n^2$
for rational parameters $x_0$. This function shares the same singularity spectrum as the classical Riemann function
described in \cite{Jaffard1996}. Moreover, they proved that $R_{x_0}$ also exhibits multifractal behaviour when
$x_0$ is irrational.

\par In our recent work \cite{Wu-Zhan2026}, we investigated the box dimension of the graphs of generalized
Riemann-type functions $G_\delta(x)=\sum_{n=1}^{\infty}g(n^{2}x)n^{-1-\delta}$, where $g$ is 1-periodic continuous
and $0<\delta\le 1$, extending the classical case of $g(x)=\sin(\pi x)$ with $\delta=1$.
Under an arithmetic non-vanishing condition on the Fourier coefficients of $g$ and a mild decay assumption,
we introduced a new mechanism to prove that the lower box dimension is at least $7/4-\delta/2$, and if
$g'$ is Lipschitz, we showed that the upper box dimension is at most  $7/4-\delta/2$,
demonstrating that the upper bound depends only on the regularity of $g$. We also
gave some examples to illustrate our non-vanishing condition's deep connection with the
Prime Number Theorem and $L$-functions, and ended with several problems for further research.

\par Motivated by \cite{Wu-Zhan2026}, in this paper we study the box dimension of the graph of
the generalized Riemann-type function $G_\delta(x)=\sum_{n=1}^{\infty}g(n^{2}x)n^{-1-\delta}$
for every nonconstant real trigonometric polynomial $g$ with $0<\delta\le1$, and prove that it
equals $7/4-\delta/2$. This gives a negative answer to \cite[problem 2]{Wu-Zhan2026},
which asks that whether there exists a trigonometric polynomial $g$ satisfying the vanishing condition
and $\dim_{\mathrm{B}}(graph(G_{\delta}))<\frac74 -\frac{\delta}{2}$ ? The main ingredient is a rational-level non-vanishing criterion
for the lower bound of lower box dimension of the graphs of $G_\delta$. More precisely, we first
introduce the square-class chirp functional $S_{d}(a/q)$ for every square-free positive integer $d$,
which is defined as a Gauss-mean average of the Fourier coefficients of $g$,
and prove that the lower bound $\dim_B(\mathrm{graph}\,G_\delta)\ge\frac74-\frac\delta2$ under a mild decay condition of
the Fourier coefficients of $g$ and non-vanishing of
$S_{d}(a/q)$ at a single rational $a/q$. Secondly, we prove a \emph{resolution theorem}, which
states that for $g$ a nonconstant real trigonometric polynomial, the chirp functional cannot vanish at every
rational simultaneously. Consequently, this implies that $\dim_B(\mathrm{graph}\,G_\delta)=\frac74-\frac\delta2$
for every nonconstant real trigonometric polynomial $g$ with $0<\delta\le1$.
Finally, we present two guiding examples distinguish structural vanishing, detected already at level $q_0=4$,
one of which is genuinely arithmetic and related to Prime Number Theorem and the automorphic series considered in Chamizo \cite{Chamizo2004}.

\par As in \cite{Wu-Zhan2026}, throughout this article, $g$ is real and $1$-periodic with Fourier expansion
$$g(x)=\sum_{k\in\mathbb Z}C_ke(kx)\quad {\rm with}\quad e(t):=e^{2\pi i t}, \quad C_{-k}=\overline{C_k}
\quad {\rm and} \quad C_0=0.$$
Given $0<\delta\leq1$, write $c_k:=C_kk^{\delta/2}$ for each $k\ge1$.
For integers $t\in\mathbb Z$ and $q\ge1$, let
\begin{equation}\label{eq:Gauss-mean}
G(t;q):=\frac1q\sum_{n=1}^{q}e\Big(\frac{tn^2}{q}\Big)
\end{equation}
denote the quadratic Gauss mean, and for $a,q$ integers, $q\ge 1$, define the \emph{chirp functional}
\begin{equation}\label{eq:chirp}
S(a;q):=\sum_{k\ge1}c_kG(ka;q).
\end{equation}
In particular, for each $a\in\mathbb{Z}$ we have $S(a;1)=\sum_k c_k$,
which corresponds to the classical level-one sum. More generally, for a square-free integer $d\ge 1$ we define the
\emph{square-class chirp functional at level $q$} by
\begin{equation}\label{eq:chirp-squareclass}
S_d(a;q):=\sum_{m\ge 1}c_{m^2d}\,G(m^2da;q).
\end{equation}
At level $q=1$ this recovers the classical square-class sum
$$S_d:=S_d(a;1)=\sum_{m\ge 1}c_{m^2d},$$
which can be viewed as base series  defining the vanishing condition in \cite{Wu-Zhan2026}.

\par The strategy of this paper is to generalize the single-level non-vanishing condition defined in
\cite{Wu-Zhan2026} to the square-class chirp functional $S_{d_0}(a_0;q_0)$ defined at an arbitrary
rational point $a_0/q_0$. In the proof of Resolution Theorem \ref{thm:resolution},
the chirp functional $S(a;q)$ defined in \eqref{eq:chirp} is applied to
a general finite sequence of coefficients and used to locate a modulus $q_0$
at which $S_{d_0}(a_0;q_0)\neq 0$, playing only an auxiliary role.
To prove Theorem \ref{thm:criterion}, we establish two new technical ingredients.
The first one is \emph{twisted local estimate} (Lemma \ref{lem:twisted}),
in which the oscillatory integral is perturbed by a $q_0$-period phase $e(ka_0n^2/q_0)$.
The second one is \emph{composite denominator counting lemma} (Lemma \ref{lem:composite_denom}), in which the
Farey points $a/p$ are replaced by  points $A/(q_0p)$ of composite denominator. Combing this with Theorem \ref{thm:criterion}
yields an unconditional dimension formula (Theorem \ref{thm:dichotomy}) for trigonometric
polynomial as seed functions, giving negative answer to an open question proposed in \cite{Wu-Zhan2026}.

\subsection{Main results}
Now we present our main results of the paper as follows.

\begin{theorem}[Resolution theorem]\label{thm:resolution}
Let $\{r_k\}_{k=1}^N$ be a finite complex sequence not identically zero. Then there
exist an integer $q_0\ge1$, depending only on $N$, and an integer $a_0$ with
$\gcd(a_0,q_0)=1$ such that $$S(a_0;q_0)=\sum_{k=1}^N r_k G(ka_0;q_0)\neq0.$$
\end{theorem}

\begin{theorem}[Rational-level non-vanishing criterion] \label{thm:criterion}
Let $g(x)$ be a {\rm 1}-periodic continuous real-valued function with Fourier coefficients $C_{k}$ and
$G_{\delta}(x)= \sum_{n=1}^{\infty}\frac{g(n^{2}x)}{n^{1+\delta}}$ for $0<\delta\le1$. If $\sum_{k=1}^{\infty}|C_{k}|k^{\frac{\epsilon+\delta}{2}}<\infty$ for some $\epsilon > 0$, and
there exist positive integers $a_{0}$ and $q_{0}$ with $\gcd(a_{0},q_{0})=1$ and a positive
square-free integer $d_0$ such that \begin{align}\label{non-vanishing condition}
S_{d_0}(a_{0},q_{0})=\sum\limits_{\substack{k=m^{2}d_0\\ m\in\mathbb{N}^+}}C_k\,k^{\delta/2}G(ka_{0};q_{0}) \ne 0
\end{align}
then $\dim_{\rm B}(\text{graph}(G_{\delta}))\ge\frac{7}{4}-\frac{\delta}{2}$.
\end{theorem}

\begin{theorem}\label{thm:dichotomy}
Let $0<\delta\le1$ and let $g$ be a nonconstant real trigonometric polynomial. Then
$$\dim_{\rm B}(\mathrm{graph}\,G_\delta)=\frac74-\frac\delta2.$$
In particular, for any trigonometric polynomial $g$, no vanishing among the square-class sums $S_d$ can
lower the box dimension of the graph of corresponding $G_{\delta}$.
\end{theorem}

\subsection{Structure of this paper}
{This paper is organized as follows.
Section \ref{sec:prelim} collects some notation and preliminaries for readers' convenience.
Section \ref{sec:proofs} proves our main results. We first prove the resolution theorem
(Theorem \ref{thm:resolution}) via an explicit and elementary construction of a suitable modulus $q_0$,
and then show the twisted local estimate (Lemma \ref{lem:twisted}) controlling the oscillatory sum $T_k(p)$
by Poisson summation, and a density lemma (Lemma \ref{lem:density}) isolating primes with non-vanishing
twisted chirp functional.
Combined with the composite denominator counting lemma (Lemma~\ref{lem:composite_denom}) and a tail estimate
from the Prime Number Theorem,
these yield the rational-level non-vanishing criterion (Theorem~\ref{thm:criterion}) and the unconditional
dimension formula for trigonometric polynomials (Theorem~\ref{thm:dichotomy}).
\par Section \ref{sec:examples} presents two examples distinguishing structural vanishing (detected already at level $q_0=4$)
from arithmetic vanishing (arising from Hasse-Weil $L$-functions and governed by Prime number theorem).

\section{Preliminaries and Conventions}\label{sec:prelim}
\par In the present paper, we adopt the  preliminary facts from our earlier work \cite{Wu-Zhan2026} , where some  standard number-theoretic tools such as Chinese remainder theorem, quadratic reciprocity, Gauss sums are collected for the convenience of the reader.

Throughout, constants denoted $C_\delta,A_\delta,\dots$ depend only on $\delta$ (and, where
indicated, on $g$, $\varepsilon$, or $q_0$) and may change from line to line.
For two positive quantities $A$ and $B$, $A\lesssim B$ means that there exists an absolute constant $C>0$
such that $A\le CB$, and $|A|\lesssim|B|$ is denoted by $A=O(B)$. Similarly, $A\lesssim_{\varepsilon}B$ denotes there exists
a constant $C_{\varepsilon}>0$ depending only on the parameter $\varepsilon$ such that $A\le C_{\varepsilon}B$. Also, $A\asymp B$
 means $C_1\le A\le C_2B$ with two universal constants $C_1,\, C_2>0$. We write
$e(t)=e^{2\pi it}$.

\section{Proof of the main results}\label{sec:proofs}
For integers $t$ and $q\ge1$, let $G(t;q)$ be defined as in \eqref{eq:Gauss-mean}.
We now begin to restate and prove Theorem \ref{thm:resolution} as follows.
\begin{theorem}{\rm(Theorem \ref{thm:resolution})}
Let $\{r_k\}_{k=1}^N$ be a finite complex sequence not identically zero. Then there
exist an integer $q_0\ge1$, depending only on $N$, and an integer $a_0$ with
$\gcd(a_0,q_0)=1$ such that $$S(a_0;q_0)=\sum_{k=1}^N r_k G(ka_0;q_0)\neq0.$$
\end{theorem}

\begin{proof}
\textbf{Step 1.} Assume $q \ge 2$. We compute the discrete Fourier transform of $S(a,q)$ over $\mathbb{Z}/q\mathbb{Z}$.
For $l\in\mathbb{Z}$, define
\begin{align}
F_{q}(l) &:= \sum_{a=1}^{q}S(a,q)e\left(-\frac{al}{q}\right) \notag \\
&= \sum_{k=1}^{N}c_{k}\left[\sum_{n=1}^{q}\frac{1}{q}\sum_{a=1}^{q}e\left(\frac{a(kn^{2}-l)}{q}\right)\right]. \label{eq:Fq_l} 
\end{align}
Note that the inner sum evaluates to
$$
\frac{1}{q}\sum_{a=1}^{q}e\left(\frac{a(kn^{2}-l)}{q}\right) =
\begin{cases}
1, & \text{if } kn^{2} \equiv l \pmod{q}, \\
0, & \text{if } kn^{2} \not\equiv l \pmod{q}.
\end{cases}
$$
Thus, for fixed integers $k,\,l$ and $q$, the expression inside the brackets in \eqref{eq:Fq_l} equals the number of solutions to the congruence $kx^{2} \equiv l \pmod{q}$ with $x \in \{1,2,\dots,q\}$. Denoting this number of solutions by $N_{q}(k,l)$, equation \eqref{eq:Fq_l} simplifies to 
\begin{equation} \label{eq:Fq_simplified}
F_{q}(l) = \sum_{k=1}^{N}c_{k}N_{q}(k,l).
\end{equation} 

\textbf{Step 2.} Since the coefficients $c_{k}$ are not all zero, let $k_{0}=\min\{1\le k\le N:\ C_k\neq0\}$.
We now set $l=k_{0}$. In the following we will construct a modulus $q$ such that the congruence
$kx^{2} \equiv k_{0} \pmod{q}$ has solutions only when $k=k_{0}$, and no solution for any $k>k_{0}$ with $c_{k}\ne0$.

\par For each fixed $k>k_{0}$, there exists a prime $p_{k}$ such that $v_{p_{k}}(k) > v_{p_{k}}(k_{0})$, where $v_{p}(m)$ denotes the exponent of $p$ in the prime factorization of $m$. Set $e_{k} = v_{p_{k}}(k_{0})+1$ and $q_{k} = p_{k}^{e_{k}}$.
Then the congruence $kx^{2} \equiv k_{0} \pmod{q_{k}}$ has no solution. Indeed, for any integer $x$, since $q_{k} \mid k$, it follows that $q_{k} \mid kx^{2}$; however, $q_{k} \nmid k_{0}$, which implies $q_{k} \nmid (kx^{2}-k_{0})$.

\par Now, choose
$$ q=\lcm\{q_{k}:\ k_0<k\le N\ \text{and}\ c_{k}\ne 0\}.$$
Since $kx^{2} \equiv k_{0} \pmod{q_{k}}$ has no solution for any $k>k_{0}$, the congruence $kx^{2} \equiv k_{0} \pmod{q}$ likewise has no solution. That is, $N_{q}(k,k_{0})=0$ for all $k>k_{0}$. 
On the other hand, for $k=k_{0}$, the congruence $k_{0}x^{2} \equiv k_{0} \pmod{q}$ clearly admits at least one solution $x=1$, so $N_{q}(k_{0},k_{0}) \ge 1$.

\par Substituting these into \eqref{eq:Fq_simplified} yields
$$ F_{q}(k_{0}) = \sum_{k=1}^{N}c_{k}N_{q}(k,k_{0}) = c_{k_{0}}N_{q}(k_{0},k_{0}) \ne 0. $$

By the initial definition \eqref{eq:Fq_l}, we also have
\begin{align*}
F_{q}(k_{0}) &= \sum_{a=1}^{q}S(a,q)e\left(-\frac{ak_{0}}{q}\right) \\
&= \sum_{a=1}^{q-1}S(a,q)e\left(-\frac{ak_{0}}{q}\right) + S(q,q) \\
&= \sum_{a=1}^{q-1}S(a,q)e\left(-\frac{ak_{0}}{q}\right) + \sum_{k=1}^{N}c_{k} \ne 0.
\end{align*} 

\par We consider following two cases.
\begin{itemize}
    \item \textbf{Case 1.} If $\sum_{k=1}^{N}c_{k} \ne 0$, we simply take $a_{0}=q_{0}=1$, and the result holds.
    \item \textbf{Case 2.} If $\sum_{k=1}^{N}c_{k} = 0$, then
    \begin{equation} \label{eq:case2}
    \sum_{a=1}^{q-1}S(a,q)e\left(-\frac{ak_{0}}{q}\right)\ne0.
    \end{equation}
\end{itemize}
Hence, there must exist $a \in \{1,2,\dots,q-1\}$ such that $S(a,q) \ne 0$. Let $d = \gcd(a,q)$,
and write $a = da_{0}$ and $q = dq_{0}$. Then $\gcd(a_{0},q_{0})=1$. For each $1\le k\le N$, we have
\begin{eqnarray} G(ak;q)=\frac{1}{dq}\sum_{n=1}^{q}e\left(\frac{kan^{2}}{q}\right)
&=& \frac{1}{dq_0}\sum_{n=1}^{dq_0}e\left(\frac{ka_0n^{2}}{q_0}\right)\\ \nonumber
&=& \frac{1}{dq_0}\sum_{r=1}^{q_0}\sum_{t=0}^{d-1} e\left(\frac{ka_0(r+tq_0)^2}{q_0}\right)\\ \nonumber
&=&\frac{1}{dq_{0}}\,d\sum_{r=1}^{q_{0}}e\left(\frac{ka_{0}r^{2}}{q_{0}}\right)=G(a_{0}k;q_{0}).\end{eqnarray} 
Thus, $S(a_{0},q_{0}) = S(a,q) \ne 0$, which completes the proof.
\end{proof}

\begin{lemma}[Twisted Local Estimate] \label{lem:twisted}
Let $p$ be a prime, $k\in\mathbb{Z}\setminus\{0\}$, $q_{0}\in\mathbb{N}^+$ and $a_0\in\mathbb{Z}$
with $\gcd(a_{0},q_{0})=1$. Given $0<\delta\le1$, define
$$T_{k}(p):=\sum_{n=1}^{\infty}\frac{e(ka_{0}n^{2}/q_{0})(1-e(kn^{2}/p^{2}))}{n^{1+\delta}}, \quad
M_{k}=\int_{0}^{\infty}\frac{1-e(ku^{2})}{u^{1+\delta}}du.$$
Set $E_k(p)=p^{\delta}T_{k}(p)-G(ka_{0};q_{0})M_{k}$. Then there exists a constant $C_{\delta}>0$
depending only on $\delta$ such that for any $0<\alpha\le1$ we have
$$ |E_{k}(p)|\le C_{\delta}\,q_{0}\,|k|^{\frac{\alpha+\delta}{2}}p^{-\alpha}.$$
\end{lemma}

\begin{proof}
Without loss of generality, we may assume $k>0$ since $E_{-k}(p)=\overline{E_k(p)}$. Define 
$$ h_{k}(u)=
\begin{cases}
\frac{1-e(ku^{2})}{|u|^{1+\delta}}, & \text{if } u \ne 0, \\
0, & \text{if } u = 0 \text{ and } 0 < \delta < 1, \\
-2\pi ik, & \text{if } u = 0 \text{ and } \delta = 1.
\end{cases}
$$ 
Since $h_k(u)$ is even, we have  $\widehat{h}_{k}(\xi) = 2k^{\frac{\delta}{2}}\Phi\left(\frac{\xi}{\sqrt{k}}\right)$, where
$$\Phi(\eta)=\int_{0}^{\infty}\frac{1-e(v^{2})}{v^{1+\delta}}\cos(2\pi\eta v)dv.$$
By \cite[Lemma 3.1]{Wu-Zhan2026}, there exist constants $A_\delta$ and $\beta > 1$ depending only on $\delta$ such that
$|\Phi(\eta)|\le A_{\delta}\eta^{-\beta}$ for $\eta\ge1$, and $|\Phi(\eta)|\le\widetilde{A}_{\delta}$ for $\eta > 0$.

\par Since $P(n):=e(ka_{0}n^{2}/q_{0})$ is an even function with period $q_{0}$, its finite Fourier expansion is given by 
$$ P(n) = \sum_{t=0}^{q_{0}-1}\hat{c}(t)e\left(\frac{tn}{q_{0}}\right)\quad {\rm with}\quad \widehat{c}\,(t)= \frac{1}{q_{0}}\sum_{\beta=0}^{q_{0}-1}P(\beta)e\left(-\frac{t\beta}{q_{0}}\right). $$ 
It follows that
$$\widehat{c}\,(0) = \frac{1}{q_{0}}\sum_{\beta=0}^{q_{0}-1}e\left(\frac{ka_{0}\beta^{2}}{q_{0}}\right) = G(ka_{0};q_{0}), $$ 
and $|\widehat{c}\,(t)|\le1$ for all $t \in \{0,1,\dots,q_{0}-1\}$. 

\par On the other hand, writing $p^{1+\delta}T_{k}(p) = \sum_{n=1}^{\infty}e(ka_{0}n^{2}/q_{0})h_{k}(n/p)$, we have 
\begin{align}
2p^{1+\delta}T_{k}(p) + h_{k}(0) &= \sum_{n \in \mathbb{Z}}P(n)h_{k}\left(\frac{n}{p}\right) \notag \\
&= \sum_{t=0}^{q_{0}-1}\widehat{c}\,(t)\sum_{n \in \mathbb{Z}}e\left(\frac{tn}{q_{0}}\right)h_{k}\left(\frac{n}{p}\right). \label{eq:Poisson_prep}
\end{align}
Applying the Poisson summation formula yields
$$ \sum_{n \in \mathbb{Z}}e\left(\frac{tn}{q_{0}}\right)h_{k}\left(\frac{n}{p}\right)=\sum_{m \in \mathbb{Z}}p\,\widehat{h}_{k}\left[p\,\left(m-\frac{t}{q_{0}}\right)\right]. $$ 
Substituting this back into \eqref{eq:Poisson_prep} gives
\begin{align}
p^{\delta}\,T_{k}(p) + \frac{h_{k}(0)}{2p} &= \frac{1}{2}\sum_{t=0}^{q_{0}-1}\widehat{c}\,(t)\sum_{m \in \mathbb{Z}}\widehat{h}_{k}\left[p\left(m-\frac{t}{q_{0}}\right)\right] \notag \\
&=\frac{1}{2}\widehat{c}\,(0)\widehat{h}_{k}(0) + \frac{1}{2}\sum_{\substack{(t,m)\ne(0,0)\\0\le t\le q_0-1,\,m\in \mathbb{Z}}}\widehat{c}\,(t)\,\widehat{h}_{k}\left[p\left(m-\frac{t}{q_{0}}\right)\right] \notag \\
&= G(ka_{0};q_{0})M_{k}+\frac{1}{2}\sum_{\substack{(t,m)\ne(0,0)\\0\le t\le q_0-1,\,m\in \mathbb{Z}}}\widehat{c}\,(t)\widehat{h}_{k}\left[p\left(m-\frac{t}{q_{0}}\right)\right]. \label{eq:error_sum}
\end{align}

\par Since $|q_0m-t|\ge1$ for each pair $(t,m)\ne(0,0)$ with $0\le t\le q_0-1$ and $m\in\mathbb{Z}$, we
can define $E_{j}=\{(t,m):\ |q_{0}m - t| = j\}$ for each $j\in\mathbb{N}^+$.
Each $j$ corresponds to at most two pairs $(t,m) \in E_{j}$. By $|\widehat{c}\,(t)| \le 1$ we obtain
\begin{align}
\left|\sum_{(t,m) \ne (0,0)}\widehat{c}\,(t)\widehat{h}_{k}\left[p\left(m-\frac{t}{q_{0}}\right)\right]\right|
&\le\sum_{\substack{(t,m)\ne(0,0)\\0\le t\le q_0-1,\,m\in \mathbb{Z}}}\left|\widehat{h}_{k}\left[p
\left(m-\frac{t}{q_{0}}\right)\right]\right| \notag \\
&= \sum_{j=1}^{\infty}\sum_{(t,m) \in E_{j}}\left|\widehat{h}_{k}\left[p\left(m-\frac{t}{q_{0}}\right)\right]\right| \notag \\
&\le 2\sum_{j=1}^{\infty}\left|\widehat{h}_{k}\left(\frac{pj}{q_{0}}\right)\right| \notag \\
&= 4k^{\frac{\delta}{2}}\sum_{j=1}^{\infty}\left|\Phi\left(\frac{pj}{q_{0}\sqrt{k}}\right)\right|. \label{eq:error_bound}
\end{align}

\par By the decay estimates of $\Phi(\eta)$, we split the sum as follows.
\begin{align*}
\sum_{j=1}^{\infty}\left|\Phi\left(\frac{pj}{q_{0}\sqrt{k}}\right)\right| &= \sum_{j \le \frac{q_{0}\sqrt{k}}{p}}\left|\Phi\left(\frac{pj}{q_{0}\sqrt{k}}\right)\right| + \sum_{j > \frac{q_{0}\sqrt{k}}{p}}\left|\Phi\left(\frac{pj}{q_{0}\sqrt{k}}\right)\right| \\
&\le \widetilde{A}_{\delta}\frac{q_{0}\sqrt{k}}{p} + A_{\delta}\sum_{j > \frac{q_{0}\sqrt{k}}{p}}\left(\frac{pj}{q_{0}\sqrt{k}}\right)^{-\beta} \\
&\le \left(\widetilde{A}_{\delta} + \frac{A_{\delta}}{\beta-1}\right)\frac{q_{0}\sqrt{k}}{p}.
\end{align*} 

\par Combining this with \eqref{eq:error_sum} and \eqref{eq:error_bound} yields
\begin{align}
|E_k(p)|=|p^{\delta}T_{k}(p) - G(ka_{0};q_{0})M_{k}| &\le 2\left(\widetilde{A}_{\delta} + \frac{A_{\delta}}{\beta-1}\right)q_{0}k^{\frac{1+\delta}{2}}p^{-1} + \frac{|h_{k}(0)|}{2p} \notag \\
&\le 2\left(\widetilde{A}_{\delta} + \frac{A_{\delta}}{\beta-1} + \pi\right)q_{0}k^{\frac{1+\delta}{2}}p^{-1}. \label{eq:T_k_bound1}
\end{align} 

\par On the other hand, direct calculation gives
\begin{align} \label{eq:T_k_direct}
|T_{k}(p)| &= \left|\sum_{n=1}^{\infty}\frac{e(ka_{0}n^{2}/q_{0})(1-e(kn^{2}/p^{2}))}{n^{1+\delta}}\right| \notag \\
&\le \sum_{n=1}^{\infty}\frac{|1-e(kn^{2}/p^{2})|}{n^{1+\delta}} \notag \\
&= \sum_{n \le p/\sqrt{k}}\frac{|1-e(kn^{2}/p^{2})|}{n^{1+\delta}} + \sum_{n > p/\sqrt{k}}\frac{|1-e(kn^{2}/p^{2})|}{n^{1+\delta}} \notag \\
&\le \sum_{n \le p/\sqrt{k}}\frac{2\pi kn^{2}/p^{2}}{n^{1+\delta}} + \sum_{n > p/\sqrt{k}}\frac{2}{n^{1+\delta}} \notag \\
&\le\frac{2\pi}{2-\delta}\,p^{-\delta}k^{\frac{\delta}{2}}+\frac{2}{\delta}\,p^{-\delta}k^{\frac{\delta}{2}}\notag \\
&:=\widetilde{C}_{\delta}p^{-\delta}k^{\frac{\delta}{2}}.
\end{align} 

\par This, together with $|M_{k}|=k^{\frac{\delta}{2}}|M_{1}|$ and \eqref{eq:T_k_direct}, implies
\begin{align}
|E_k(p)|=|p^{\delta}T_{k}(p)-G(ka_{0};q_{0})M_{k}| &\le p^{\delta}|T_{k}(p)| + |M_{k}| \notag \\
&\le (\widetilde{C}_{\delta}+|M_{1}|)k^{\frac{\delta}{2}}. \label{eq:T_k_bound2}
\end{align} 
\par Taking the weighted geometric mean of the bounds \eqref{eq:T_k_bound1} and \eqref{eq:T_k_bound2} with weights $\alpha$
and $1-\alpha$ respectively, we concluded that for any $0 < \delta \le 1$,
$$ |E_{k}(p)| = |p^{\delta}T_{k}(p) - G(ka_{0};q_{0})M_{k}| \le C_{\delta}q_{0}^{\alpha}k^{\frac{\alpha+\delta}{2}}p^{-\alpha}\le C_{\delta}q_{0}k^{\frac{\alpha+\delta}{2}}p^{-\alpha}. $$
\end{proof}

\begin{lemma} \label{freezing-lem} Given a sequence \(\{\lambda_k\}_{k\in\mathbb{Z}}\) satisfying $\lambda_0=0,\ \lambda_k=\overline{\lambda_{-k}}$ and $\sum_{k\in\mathbb{Z}}|\lambda_k|<\infty$.
Let \(D^{+}\) be the set of all square-free positive integers.
Soppose  $S_{d_0}(a_{0},q_{0}):=\sum_{m=1}^{\infty}\lambda_{m^2 d_0}G(a_{0}m^2 d_0;q_{0}) \ne 0$ for some \(d_0\in D^{+}\) and positive integers $a_0$, $q_0$ with $\gcd(a_{0},q_{0})=1$. Set
$$ \Theta_{p} = \sum_{k\in\mathbb{Z}}\lambda_{k}G(ka_{0};q_{0})\left(\frac{k}{p}\right). $$
Then there exist positive integers $b$ and $B$ satisfying $8q_{0} \mid B$, $\gcd(b,B)=1$ and a set of odd primes
$$ \mathcal{P}=\{\,p\ \text{odd prime}:\ p\equiv b\,({\rm mod}\ B)\}$$ 
such that $$|\Theta_{p}|\ge\eta$$
holds for all primes $p \in \mathcal{P}$, where $\eta > 0$ depends only on the sequence $\{\lambda_{k}\}_{k\in\mathbb{Z}}$.
\end{lemma}

\begin{proof} Let $r_k=\lambda_{k}G(ka_{0};q_{0})$, then $\sum_{k\in\mathbb{Z}}|r_k|\le\sum_{k\in\mathbb{Z}}|\lambda_k|<\infty$.
Applying the argument in the proof of \cite[Lemma 3.3]{Wu-Zhan2026} to $\{r_k\}_{k\in\mathbb{Z}}$ derives the conclusion.
\end{proof}

\begin{lemma}[Composite Denominator Counting Lemma] \label{lem:composite_denom}
 For prime $p$, denote $R(p)$ by the quadratic residue modulo $p$.  Let $a_0,\,q_0$ and $b$ be the same integers as in Lemma
{\rm \ref{freezing-lem}} with $\gcd(a_{0},q_{0})=1$. Fix $0<C_1<C_2$, for a sufficiently large integer $N$ define
$$\mathcal{P}_{N}:=\{\,p\ \text{prime}:\ C_{1}\sqrt{N} \le p \le C_{2}\sqrt{N}\}, \quad D_{p,a}:=\left[x_{A}, x_{A} + \frac{1}{p^{2}}\right]$$
with $x_{A}=\frac{A}{q_{0}p}\in (0,1)$, where $A$ depends only on $a_0 b$ and $a$ satisfying $A\equiv a_{0}b \pmod{q_{0}}$, $A \equiv a \pmod{p}$
with $a\in R(p)$. Then for any continuous function $F:\ [0,1]\to\mathbb{R}$, the box dimension of its graph satisfies
\begin{align*} \underline{\dim}_{\rm B}(\text{graph } F) \ge \varliminf_{N \to\infty}\frac{\log\left(N\sum_{p \in \mathcal{P}_{N}}\sum_{a \in R(p)}\osc(F;D_{p,a})\right)}{\log N}.
\end{align*}
\end{lemma}
\begin{proof}
\textbf{Step 1 (Bounding the number of $D_{p,a}$ meeting each partition interval).} Let $\mathcal{D} = \{D_{p,a} : p \in \mathcal{P}_{N}, a \in R(p)\}$, and set $I_{k} = \left[\frac{k}{N}, \frac{k+1}{N}\right]$ for $0 \le k \le N-1$. We claim that there exists a constant $\tilde{C}$ depending only on $C_{1}$ and $C_{2}$ such that for each $k$, at most $\tilde{C}$ intervals in $\mathcal{D}$ intersect $I_{k}$.

\par Indeed, each interval $D_{p,a} \in \mathcal{D}$ has length $\frac{1}{p^{2}} \in [\delta,\Delta]$, where
$$\delta:=\frac{1}{q_0\,C_{2}^{2}N}, \quad \Delta := \frac{1}{C_{1}^{2}N}. $$ 
Notice that if $\frac{A}{q_{0}p}\ne\frac{A'}{q_{0}p'}$ with $p,\,p'\in\mathcal{P}_N$, then using $|Ap'-A'p|\ge1$ and
$p,\,p'\le C_{2}\sqrt{N}$ obtains
\begin{equation}\label{delta}\left|\frac{A}{q_{0}p}-\frac{A'}{q_{0}p'}\right|=\frac{|Ap'-A'p|}{q_{0}pp'}\ge\frac{1}{q_{0}pp'}\ge \frac{1}{q_{0}C_{2}^{2}N}:=\delta.\end{equation} 
Thus, if $D_{p,a}=\left[x_{A}, x_{A}+\frac{1}{p^{2}}\right]$ intersects $I_{k}$, then $x_{A} \le \frac{k+1}{N}$ and $x_{A} + \frac{1}{p^{2}} \ge \frac{k}{N}$, which implies 
$$ x_{A} \in \left[\frac{k}{N} - \frac{1}{p^{2}}, \frac{k+1}{N}\right] \subseteq \left[\frac{k}{N} - \frac{1}{C_{1}^{2}N}, \frac{k+1}{N}\right] := J_{k}. $$ 
The length of $J_{k}$ is $\frac{1}{N}\left(\frac{1}{C_{1}^{2}}+1\right):= l$. By \eqref{delta}, any two distinct left endpoints $x_{A} = \frac{A}{q_{0}p}$ and $x_{A'} = \frac{A'}{q_{0}p'}$ corresponding to different intervals $D_{p,a},D_{p',a'} \in \mathcal{D}$ satisfy $|x_{A} - x_{A'}| \ge \delta$, we have 
\begin{equation} \label{eq:Cwidetilde}
\text{card}\{D_{p,a} \in \mathcal{D} : D_{p,a} \cap I_{k} \ne \emptyset\} \le \left[\frac{l}{\delta}\right] + 1 = \left[\frac{q_{0}C_{2}^{2}(1+C_{1}^{2})}{C_{1}^{2}}\right] + 1:=\widetilde{C}.
\end{equation} 

\textbf{Step 2 (Lower bound on the box dimension).} Denote by $A_{N}(F)$ the minimum number
of boxes of size $1/N$ required to cover $\text{graph}(F)$. Then
\begin{equation} \label{eq:box_cover}
A_{N}(F) \ge N\sum_{k=0}^{N-1}\osc(F;I_{k}),
\end{equation} 
where $\osc(F;I) = \sup_{x \in I}F(x) - \inf_{x \in I}F(x)$. Since $D_{p,a} \subseteq [0,1]$, we have 
$$ \osc(F;D_{p,a}) \le \sum_{k : I_{k} \cap D_{p,a} \ne \emptyset}\osc(F;I_{k}). $$ 
Applying \eqref{eq:Cwidetilde} gains
$$ \sum_{p \in \mathcal{P}_{N}}\sum_{a \in R(p)}\osc(F;D_{p,a}) \le \widetilde{C}\sum_{k=0}^{N-1}\osc(F;I_{k}). $$ 
Combining this with \eqref{eq:box_cover}, we obtain $A_{N}(F) \ge \frac{N}{\widetilde{C}}\sum_{p \in \mathcal{P}_{N}}\sum_{a \in R(p)}\osc(F;D_{p,a})$. Since $\widetilde{C}$ is independent of $N$, the lower box dimension satisfies 
$$ \underline{\dim}_{\rm B}(\text{graph } F) = \varliminf_{N \to\infty}\frac{\log A_{N}(F)}{\log N} \ge \varliminf_{N \to\infty}\frac{\log\left(N\sum_{p \in \mathcal{P}_{N}}\sum_{a \in R(p)}\osc(F;D_{p,a})\right)}{\log N}. $$
\end{proof}

The following Lemma is crucial for estimating the error term in {\rm Theorem \ref{thm:criterion}} and can be found in \cite{Wu-Zhan2026}. We give its detailed
proof for readers' convenience.
\begin{lemma} \label{lem:density}
Let $\mathbb{P}$ be the set of all primes. Suppose $\{a_{k}\}_{k \ge 1} \subset \mathbb{C}$ satisfy $\sum_{k=1}^{\infty}|a_{k}|k^{\beta} < \infty$ for some constant $\beta > 0$. Then for any subset $\mathcal{P} \subset \mathbb{P}$ with natural density $d_{\mathbb{P}}(\mathcal{P}) > 0$ and any constant $C > 0$, the set 
$$ E_{C} := \left\{p \in \mathcal{P} : \sum_{p \mid k}|a_{k}| \le C p^{-\beta-1}\right\} $$ 
has natural density $d_{\mathcal{P}}(E_{C}) = 1$. 
\end{lemma}

\begin{proof}
Denote $M = \sum_{k=1}^{\infty}|a_{k}|k^{\beta}$, and for each prime $p$, define the weighted sum $S_{p} = \sum_{p \mid k}|a_{k}|$. 
Then
\begin{equation} \label{eq:W_sum}
W := \sum_{p \in \mathcal{P}}p^{\beta}S_{p} = \sum_{p \in \mathcal{P}}p^{\beta}\sum_{p \mid k}|a_{k}| = \sum_{k=1}^{\infty}|a_{k}|\sum_{p \in \mathcal{P}, \, p \mid k}p^{\beta}.
\end{equation} 

\textbf{Claim:} There exists a constant $C_{\beta} > 0$ depending only on $\beta$ such that for all $k \ge 2$ we have 
\begin{equation} \label{eq:C_beta}
\sum_{p \in \mathbb{P}, \, p \mid k}p^{\beta} \le C_{\beta}k^{\beta}.
\end{equation} 

\par Indeed, each $k$ can be factorized as $k = p_{1}^{\alpha_{1}}p_{2}^{\alpha_{2}}\dots p_{s}^{\alpha_{s}}$ with primes $p_{1} < p_{2} < \dots < p_{s}$ and $\alpha_{i} \ge 1 \, (1 \le i \le s)$, which implies 
$$ \sum_{p \in \mathbb{P}, \, p \mid k}p^{\beta} = \sum_{i=1}^{s}p_{i}^{\beta}. $$ 
This, together with $k^{\beta} \ge (p_{1}p_{2}\dots p_{s})^{\beta}$ since $\alpha_{i} \ge 1$, gives 
\begin{equation} \label{eq:R_k}
R(k) := \frac{1}{k^{\beta}}\sum_{p \in \mathbb{P}, \, p \mid k}p^{\beta} \le \frac{\sum_{i=1}^{s}p_{i}^{\beta}}{\prod_{j=1}^{s}p_{j}^{\beta}} = \sum_{i=1}^{s}\prod_{j \ne i}p_{j}^{-\beta}.
\end{equation} 
Let $\{q_{n}\}_{n \ge 1}$ label the sequence of all primes ordered increasingly. Since $p_{n} \ge q_{n} \ge n$ and $q_{n} \sim n\log n$ as $n \to \infty$, by \eqref{eq:R_k} we derive 
$$ R(k) \le \sum_{i=1}^{s}\prod_{j \ne i}q_{j}^{-\beta} = \frac{\sum_{i=1}^{s}q_{i}^{\beta}}{\prod_{j=1}^{s}q_{j}^{\beta}} \le \frac{s^{\beta+1}(\log s)^{\beta}}{(s!)^{\beta}}. $$ 
Combining this with $\lim_{s \to \infty}\frac{s^{\beta+1}(\log s)^{\beta}}{(s!)^{\beta}} = 0$ guarantees that there must exist a constant $C_{\beta} > 0$ depending only on $\beta$ such that $R(k) \le C_{\beta}$ for all $k$, so the Claim \eqref{eq:C_beta} holds. 

\par Substituting \eqref{eq:C_beta} into \eqref{eq:W_sum} yields 
$$ W = \sum_{p \in \mathcal{P}}p^{\beta}S_{p} \le \sum_{k=1}^{\infty}|a_{k}|C_{\beta}k^{\beta} = C_{\beta}M < \infty. $$ 
Thus $\sum_{p \in \mathcal{P}}\frac{1}{p}(p^{1+\beta}S_{p}) < \infty$. Set $\lambda_{p} := p^{1+\beta}S_{p}$, then $\sum_{p \in \mathcal{P}}\frac{1}{p}\lambda_{p} < \infty$. 

\par Suppose the Lemma is false. Then there exist a subset $\mathcal{P} \subset \mathbb{P}$ with natural density $d_{\mathbb{P}}(\mathcal{P}) > 0$ and a constant $C > 0$ such that the set 
$$ E_{C} := \left\{p \in \mathcal{P} : \sum_{p \mid k}|a_{k}| \le C p^{-\beta-1}\right\} = \{p \in \mathcal{P} : \lambda_{p} \le C\} $$ 
has natural density $d_{\mathcal{P}}(E_{C}) < 1$. So its complement 
$$ G_{C} := \mathcal{P} \setminus E_{C} = \{p \in \mathcal{P} : \lambda_{p} > C\} $$ 
has natural density $d_{\mathcal{P}}(G_{C}) > 0$. Then by $d_{\mathbb{P}}(\mathcal{P}) > 0$, we have $d_{\mathbb{P}}(G_{C}) > 0$. This implies 
$$ \sum_{p \in G_{C}}\frac{1}{p} = \infty, $$ 
and thus
$$ \sum_{p \in G_{C}}\frac{1}{p}\lambda_{p} \ge C\sum_{p \in G_{C}}\frac{1}{p} = \infty, $$ 
which contradicts
$$ \sum_{p \in G_{C}}\frac{1}{p}\lambda_{p} \le \sum_{p \in \mathcal{P}}\frac{1}{p}\lambda_{p} < \infty. $$ 
Therefore, Lemma \ref{lem:density} follows.
\end{proof}

\par Now we restate and prove Theorem \ref{thm:criterion} step by step.
\begin{theorem} {\rm(Theorem \ref{thm:criterion})}
Let $g(x)$ be a {\rm 1}-periodic continuous real-valued function with Fourier coefficients $C_{k}$ and
$G_{\delta}(x)= \sum_{n=1}^{\infty}\frac{g(n^{2}x)}{n^{1+\delta}}$ for $0<\delta\le1$. If $\sum_{k=1}^{\infty}|C_{k}|k^{\frac{\epsilon+\delta}{2}}<\infty$ for some $\epsilon>0$, and
there exist positive integers $a_{0}$ and $q_{0}$ with $\gcd(a_{0},q_{0})=1$ and a positive
square-free integer $d_0$ such that \begin{align}\label{non-vanishing condition}
S_{d_0}(a_{0},q_{0})=\sum\limits_{\substack{k=m^{2}d_0\\ m\in\mathbb{N}^+}}C_k\,k^{\delta/2}G(ka_{0};q_{0}) \ne 0,
\end{align}
then $\dim_{\rm B}(\text{graph}(G_{\delta}))\ge\frac{7}{4}-\frac{\delta}{2}$.
\end{theorem}

\begin{proof}
\textbf{Step 1: Variation on quadratic residues.}
\par Let $M_k$ be as in Lemma \ref{lem:twisted} and set\ $\lambda_k:=C_k M_k$. For $k\geq 0$, change of variable \(u = t/\sqrt{k}\) shows that
\[M_k = \int_0^\infty \frac{1-e^{2\pi i k u^{2}}}{u^{1+\delta}}\,du
= k^{\delta/2}\int_0^\infty \frac{1-e^{2\pi i t^{2}}}{t^{1+\delta}}\,dt = k^{\delta/2} M_1.\]
 First $M_{-k}=\overline{M_k}$ implies $\lambda_{-k} = \overline{\lambda_{k}}$. On the other hand
 we have
\[\sum_{k\in\mathbb{Z}}|\lambda_k|=2|M_1|\sum_{k=1}^{\infty}|C_k| k^{\frac{\delta}{2}}
\lesssim\sum_{k=1}^{\infty}|C_k| k^{\frac{\varepsilon+\delta}{2}}<\infty.\]

\par Furthermore, combining $\lambda_k= M_1C_k k^{\delta/2}$ for $k>0$ and the non-vanishing condition (\ref{non-vanishing condition}) we obtain
$S_{d_0}(a_{0},q_{0}):=\sum_{m=1}^{\infty}\lambda_{m^2 d_0}G(a_{0}m^2 d_0;q_{0}) \ne 0$.
Thus by Lemma~\ref{freezing-lem}, there exist positive integers $b$ and $B$ satisfying $8q_{0} \mid B$, $\gcd(b,B)=1$ and a set of odd primes
\[\mathcal{P}=\{\,p\ \text{odd\ prime}:\ p\equiv b\,({\rm mod}\,B)\}\]
such that for all \(p\in\mathcal{P}\) we have
\[\Bigl|\sum_{k\in\mathbb{Z}}\lambda_{k}G(ka_{0};q_{0})\left(\frac{k}{p}\right)\Bigr|\ge \eta>0,\]
where \(\eta=\bigl(\frac{\sqrt{2}}{2}-\frac12\bigr)|S_{d_0}(a_{0},q_{0})|\) and
\(S_{d_0}(a_{0},q_{0})=\sum\limits_{\substack{k=m^{2}d_0\\ m\in\mathbb{N}^+}}C_k\,k^{\delta/2}G(ka_{0};q_{0})\).

For prime $p$, let $R(p)$ denote the quadratic residue modulo $p$.
Take $a\in R(p)$ with $p\in\mathcal{P}$. Then $\gcd(q_{0},p)=1$ because $p\equiv b \pmod{B}$, $8q_{0} \mid B$ and $\gcd(b,B)=1$.
Choose $A\in\mathbb{N}^+$ such that $A \equiv a_{0}b \pmod{q_{0}}$ and $A \equiv a \pmod{p}$. Set $x_{A} = \frac{A}{q_{0}p}$ and 
$$ \Delta(a) = G_{\delta}(x_{A}) - G_{\delta}\left(x_{A}+\frac{1}{p^{2}}\right). $$ 

\par By expanding $\Delta(a)$, we have 
\begin{equation} \label{eq:delta_plus}
\Delta(a)=\sum_{k \ne 0}C_{k}\sum_{n=1}^{\infty}\frac{e(kn^{2}x_{A})(1-e(kn^{2}/p^{2}))}{n^{1+\delta}}.
\end{equation} 

\par Because $\gcd(q_{0},p)=1$, there exist integers $\overline{p}$ and $\overline{q}_{0}$ such that $p\overline{p}\equiv1\pmod{q_0}$ and $q_0\overline{q}_0\equiv1\pmod{p}$.
Since $A \equiv a_{0}b\pmod{q_0}$ and $A \equiv a\pmod{p}$,
the Chinese Remainder Theorem (see e.g.\cite[Lemma 2.3]{Wu-Zhan2026}) gives
$$A\equiv p\overline{p}a_{0}b+q_{0}\overline{q}_{0}a\pmod{q_0p}.$$ 
Thus,
\begin{equation} \label{eq:e_kn2}
e(kn^{2}x_{A}) = e\left(\frac{kn^{2}A}{q_{0}p}\right) = e\left(\frac{kn^{2}a_{0}b\overline{p}}{q_{0}}\right)e\left(\frac{kn^{2}a\overline{q}_{0}}{p}\right).
\end{equation} 
Since $8q_{0}\mid B$ and $p\equiv b \pmod{B}$,
we have $p\equiv b\pmod{q_0}$ and so $p\overline{p}\equiv b\overline{p}\pmod{q_0}$.
Then $p\overline{p}\equiv 1\pmod{q_0}$ suggests $b\overline{p}\equiv 1\pmod{q_0}$. Hence,
\begin{equation} \label{eq:e_a0}
e\left(\frac{kn^{2}a_{0}b\overline{p}}{q_{0}}\right) = e\left(\frac{kn^{2}a_{0}}{q_{0}}\right).
\end{equation} 

\par By $q_{0}\,\overline{q}_{0}\equiv1\pmod{p}$ we have
$\left(\frac{k\overline{q}_{0}}{p}\right)=\left(\frac{k}{p}\right)\left(\frac{\overline{q}_{0}}{p}\right)
=\left(\frac{k}{p}\right)\left(\frac{q_{0}}{p}\right)$. Substituting $k\overline{q}_0$ for $k$ in
\cite[(2.3)]{Wu-Zhan2026} attains
\begin{equation} \label{eq:sum_a}
\sum_{a \in R(p)}e\left(\frac{kn^{2}a\overline{q}_{0}}{p}\right) =
\begin{cases}
\frac{p-1}{2}, & \text{if } p \mid k, \\
\frac{1}{2}\left[\varepsilon_{p}\sqrt{p}\left(\frac{k}{p}\right)\left(\frac{q_{0}}{p}\right)-1\right], & \text{if } p \nmid k \text{ and } p \nmid n.
\end{cases}
\end{equation} 
Substituting \eqref{eq:e_kn2}, \eqref{eq:e_a0}, and \eqref{eq:sum_a} into \eqref{eq:delta_plus}, we obtain 
\begin{equation} \label{eq:delta_sum_final}
\sum_{a \in R(p)}\Delta(a) = \frac{\varepsilon_{p}\sqrt{p}}{2}\left(\frac{q_{0}}{p}\right)\sum_{k \ne 0}C_{k}\left(\frac{k}{p}\right)T_{k}(p) - \frac{1}{2}\sum_{k \ne 0}C_{k}T_{k}(p) + \frac{p-1}{2}\sum_{p \mid k}C_{k}T_{k}(p).
\end{equation} 

\textbf{Step 2: Isolating the main term.} By Lemma \ref{lem:twisted}, $T_{k}(p) = p^{-\delta}(G(ka_{0};q_{0})M_{k} + E_{k}(p))$. Substituting this into \eqref{eq:delta_sum_final} gives 
\begin{align}
\sum_{a \in R(p)}\Delta(a) &= \frac{\varepsilon_{p}}{2}p^{\frac{1}{2}-\delta}\left(\frac{q_{0}}{p}\right)\sum_{k \ne 0}C_{k}M_{k}\left(\frac{k}{p}\right)G(ka_{0};q_{0}) \notag \\
&\quad + \frac{\varepsilon_{p}}{2}p^{\frac{1}{2}-\delta}\left(\frac{q_{0}}{p}\right)\sum_{k \ne 0}C_{k}\left(\frac{k}{p}\right)E_{k}(p) - \frac{1}{2}\sum_{k \ne 0}C_{k}T_{k}(p) + \frac{p-1}{2}\sum_{p \mid k}C_{k}T_{k}(p) \notag \\
&:=I_1 + I_{2} + I_{3} + I_{4}.\label{eq:Lambda_def}
\end{align} 

\par Also, since $\left(\frac{k}{p}\right)=0$ if $p\mid k$, it follows by Step 1 that
\[|I_1| = \frac12 p^{1/2-\delta} \Bigl|\sum_{\substack{k\in\mathbb{Z}\\p\nmid k}}
\lambda_{k}G(ka_{0};q_{0})\left(\frac{k}{p}\right)\Bigr|
=\frac12 p^{1/2-\delta} \Bigl|\sum_{k\in\mathbb{Z}}\lambda_{k}G(ka_{0};q_{0})\left(\frac{k}{p}\right)\Bigr|\gtrsim p^{1/2-\delta}.
 \]

\textbf{Step 3: Bounding the error terms.} Applying Lemma \ref{lem:twisted} with $\alpha=\epsilon$ to $I_{2}$ yields
$$|I_{2}|\le\frac{1}{2}\,p^{\frac{1}{2}-\delta}\sum_{k\ne0}|C_{k}||E_{k}(p)|
\le\frac{1}{2}\,p^{\frac{1}{2}-\delta}\sum_{k\ne0}|C_{k}|C_\delta\,q_0|k|^{\frac{\epsilon+\delta}{2}}p^{-\epsilon}\le C_{\delta}^{'}p^{\frac{1}{2}-\delta-\epsilon}$$
with $C_{\delta}^{'}=C_\delta\,q_0\sum_{k=1}^{\infty}|C_{k}|k^{\frac{\epsilon+\delta}{2}}<\infty$. 

\par By \eqref{eq:T_k_direct} we have 
$$ |I_{3}| \le \frac{1}{2}\sum_{k \ne 0}|C_{k}||T_{k}(p)| \le C_{\delta}p^{-\delta}\sum_{k\ne0}|C_{k}||k|^{\frac{\delta}{2}} \le C_{\delta}^{''}p^{-\delta}$$ with $C_{\delta}^{''}=2\,C_{\delta}\sum_{k=1}^{\infty}|C_k|k^{\frac{\epsilon+\delta}{2}}<\infty$.

\par For $I_{4}$, since $\sum_{k=1}^{\infty}|C_{k}|k^{\frac{\epsilon+\delta}{2}} < \infty$ and $d_{\mathbb{P}}(\mathcal{P}) = \frac{1}{\varphi(B)} > 0$, applying Lemma \ref{lem:density} with $a_{k} = C_{k}$, $\beta = \frac{\epsilon+\delta}{2}$, and $C = 1$ yields a subset $E_{1} \subseteq \mathcal{P}$ with $d_{\mathcal{P}}(E_{1}) = 1$ such that 
$$ \sum_{p \mid k}|C_{k}| \le p^{-\frac{\epsilon+\delta}{2}-1} $$ 
for all $p \in E_{1}$. This, together with $|T_{k}(p)|\le2\zeta(1+\delta)$, yields 
$$ |I_{4}| \le p\left|\sum_{p \mid k}C_{k}T_{k}(p)\right| \le C p\sum_{p \mid k}|C_{k}|\le2\zeta(1+\delta)p^{-\frac{\epsilon+\delta}{2}}. $$ 

\par Therefore, combining the above estimates, we obtain
\begin{align}
\left|\sum_{a \in R(p)}\Delta(a)\right| &\gtrsim p^{1/2-\delta} \label{eq:max_variation}
\end{align} 
for all $p \in E_{1}$ and $0 < \delta \le 1$. 

\textbf{Step 4: From variation to dimension.} Let $\mathcal{P}_{N}$ and $E_{1}$ be as defined previously. Fix a large integer $N$ and choose primes $p \in E_{1}$ with $p\asymp N^{1/2}$. As $N \to \infty$, we have 
$$ \text{card}(E_{1} \cap \mathcal{P}_{N}) = \text{card}\{p \in E_{1} : p\asymp N^{1/2}\} \sim \frac{N^{1/2}}{\varphi(B)\log(N^{1/2})}. $$ 
Recall that $R(p)$ consists of quadratic residues $\pmod{p}$. 
Combining $\sum_{a \in R(p)}\osc(G_{\delta};D_{p,a}) \ge \left|\sum_{a \in R(p)}\Delta(a)\right|$ with Lemma \ref{lem:composite_denom} and \eqref{eq:max_variation}, we have 
\begin{align*}
\dim_{B}(\text{graph}(G_{\delta})) &\ge \lim_{N \to \infty}\frac{\log\left(N\sum_{p \in \mathcal{P}_{N}}\sum_{a \in R(p)}\osc(G_{\delta};D_{p,a})\right)}{\log N} \\
&\ge \lim_{N \to \infty}\frac{\log\left(N\sum_{p \in E_{1} \cap \mathcal{P}_{N}}\left|\sum_{a \in R(p)}\Delta(a)\right|\right)}{\log N} \\
&\ge \lim_{N \to \infty}\frac{\log\left(N\sum_{p \in E_{1} \cap \mathcal{P}_{N}}p^{\frac{1}{2}-\delta}\right)}{\log N} \\
&= \lim_{N \to \infty}\frac{\log\left(N \cdot N^{\frac{1}{4}-\frac{\delta}{2}}\text{card}(E_{1} \cap \mathcal{P}_{N})\right)}{\log N} \\
&= \frac{7}{4} - \frac{\delta}{2}.
\end{align*} 
This completes the proof of Theorem \ref{thm:criterion}. 
\end{proof}

\par It is our turn to give the proof of Theorem \ref{thm:dichotomy}.
\begin{theorem} {\rm (Theorem \ref{thm:dichotomy})}
Let $0<\delta\le1$ and let $g$ be a nonconstant real trigonometric polynomial. Then
$$\dim_{\rm B}(\mathrm{graph}\,G_\delta)=\frac74-\frac\delta2.$$
In particular, for any trigonometric polynomial $g$, no vanishing among the square-class sums $S_d$ can
lower the box dimension of the graph of corresponding $G_{\delta}$.
\end{theorem}

\begin{proof}
Let $g(x)=\sum_{|k|\le N}C_k e^{2\pi i k x}$ with $C_{-k}=\overline{C_k}$.
Since $g(x)$ is a nonconstant trigonometric polynomial, not all $C_k$ ($1\le k\le N$) vanish.
Set $r_k=C_k k^{\frac{\delta}{2}}$ for $1\le k\le N$.
Then not all $\{r_k\}_{k=1}^N$ are zero.
By Theorem 1.1, there exist positive integers $a_0,q_0$ with $\gcd(a_0,q_0)=1$ such that
\[S(a_0,q_0)=\sum_{k=1}^N C_k k^{\frac{\delta}{2}}G(ka_0;q_0)\neq 0.\]

\par Observe that
\[\sum_{k=1}^N C_k k^{\frac{\delta}{2}} G(ka_0;q_0)=\sum_{1\le d\le N}\;
\sum_{\substack{k=m^2 d\\m\ge 1}} C_k k^{\frac{\delta}{2}} G(ka_0;q_0),\]
where $d$ denotes square-free positive integers.
Hence there exists some $d_0 \geq 1$ such that
\[\sum_{\substack{k=m^2 d_0\\m\ge 1}} C_k k^{\frac{\delta}{2}} G(ka_0;q_0)\neq 0.\]
It then follows from Theorem \ref{thm:criterion} that
\[\dim_B(\operatorname{graph}(G_\delta))\ge \frac{7}{4}-\frac{\delta}{2}.\]

\par On the other hand, $g'(x)$ is a Lipschitz continuous function on $\mathbb{R}$ with period $1$.
Applying \cite[Theorem 1.3]{Wu-Zhan2026}, we obtain
\[\dim_B(\operatorname{graph}(G_\delta))\le \frac{7}{4}-\frac{\delta}{2}.\]
Therefore,
\[\dim_B(\operatorname{graph}(G_\delta))=\frac{7}{4}-\frac{\delta}{2}.\]
\end{proof}

\section{Two guiding examples}\label{sec:examples}

In the following, we present two examples to illustrate the relationship between rational-level vanishing of the square-class chirp functional
$S_{d}(a;q)$ and the box dimension of the corresponding graph: the first exhibits structural vanishing,
already excluded by the resolution theorem, and the second exhibits genuinely arithmetic vanishing
governed by the Prime Number Theorem for arithmetic progressions.

\begin{example}\label{ex:translation}
{\rm Let $g(x)=2^{\delta+1}\cos(2\pi x)-2\cos(8\pi x)$ with $0<\delta\le1$
such that $C_{\pm1}=2^{\delta}$, $C_{\pm4}=-1$, $C_k=0$ for $k\not\in\{\pm1,\pm4\}$ and
denote $G_{\delta}=\sum_{n=1}^{\infty}g(n^2x)n^{-1-\delta}$.

\par Direct computation gets
\[S_{d_0}=\sum_{\substack{k=m^{2}d_0\\ m\in\mathbb{N}^+}}C_k k^{\delta/2}=\begin{cases}
C_1\cdot1^{\delta/2}+C_4\cdot4^{\delta/2}=2^{\delta}-2^{\delta}=0,&d_0=1,\\[2pt]
\ 0-0=0,&d_0\ge2,
\end{cases}\]
so the vanishing condition (see \cite{Wu-Zhan2026})  holds for each square-free $d_0\in\mathbb{N}^+$.

\par Nevertheless, writing $F_{\delta}(x)=\sum_{n=1}^{\infty}\frac{\cos(2\pi n^{2}x)}{n^{1+\delta}}$
and as shown in \cite{Wu-Zhan2026},
$G_{\delta}(x)=-2^{\delta+1}F_{\delta}(x+\tfrac12)$.
This suggests that the box dimension of $graph(G_{\delta})$   is still $\frac{7}{4}-\frac{\delta}{2}$.

\par In fact, the chirp functional detects this immediately:
$$S(1;4)=c_1G(1;4)+c_4G(4;4)=2^\delta\cdot\frac{1+i}2-2^\delta\cdot1=2^{\delta-1}(i-1)\neq0,$$
Thus by Theorem \ref{thm:dichotomy}, the box dimension of $graph(G_{\delta})$   is indeed $\frac{7}{4}-\frac{\delta}{2}$.}
\end{example}

\begin{example}\label{ex:elliptic} {\rm Following Chamizo {\rm\cite{Chamizo2004}}},
{\rm let $E/\mathbb{Q}$ be a modular elliptic curve with Hasse-Weil $L$-function
$L(s,E)=\sum\limits_{n=1}^{\infty}a_{n}n^{-s}$. Define
$G_{\delta}(x)=\sum_{n=1}^{\infty}{n^{-1-\delta}}a_n\cos(2\pi nx)$ with $1/2<\delta\le1$.
Then by Chamizo {\rm\cite[Proposition 2.4]{Chamizo2004}}, we have
$\dim_{\mathrm{B}}(\operatorname{graph}{(G_{\delta})})=2-\delta<\frac{7}{4}-\frac{\delta}{2}$.

\par As shown in \cite{Wu-Zhan2026},  $G_{\delta}$ can be expressed as
$G_{\delta}(x)=\sum_{n=1}^{\infty}n^{-1-\delta}g(n^{2}x)$ with a
$1$-periodic function $g$ having Fourier coefficients $C_{k}$, which satisfy the vanishing
condition $S_{d}(a,1)=0$ for every square-free $d\in\mathbb{N}^+$. Let $a,q\ge 1$ be positive integers with $\gcd(a,q)=1$.
For any square-free integer $d$, we aim to prove that
\begin{equation} \label{Sdaq} S_d(a,q)=\sum_{m=1}^{\infty}C_{dm^2}(dm^2)^{\delta/2}G(dm^2 a;q)=0. \end{equation}
Define the partial sum
\[T_{d,N}(a,q)=\sum_{m=1}^{N}C_{dm^2}(dm^2)^{\delta/2}G(dm^2 a;q).\]
Then $S_d(a,q)=\lim_{N\to\infty}T_{d,N}$.

\par As shown in \cite{Wu-Zhan2026},
\[
C_{dm^{2}}\,m^{1+\delta}
=\frac{1}{d^{1+\delta}}\sum_{k\mid m}\mu\,\left(\tfrac{m}{k}\right)h(k),
\]
where $h(k):=a_{dk^{2}}/(2k^{1+\delta})$.
we have
\[
T_{d,N}=\frac{1}{d^{1+\frac{\delta}{2}}}\sum_{m=1}^{N}\frac{1}{m}\sum_{k\mid m}\mu\Big(\frac{m}{k}\Big)h(k)G(dm^2 a;q).
\]
The change of variales $m=kl$ yields
\begin{equation} \label{TdN}
T_{d,N}=\frac{1}{d^{1+\frac{\delta}{2}}}\sum_{k=1}^{N}\frac{h(k)}{k}
\sum_{l=1}^{\lfloor\frac{N}{k}\rfloor}\frac{\mu(l)}{l}G(adk^2l^2;q).
\end{equation}
Write
\begin{equation}\label{Wkx} W_k(x)=\sum_{l\le x}\frac{\mu(l)}{l}G(adk^2l^2;q).\end{equation}
Let $t=adk^2$ and consider the Gauss sum
\begin{equation} \label{Fl} F(l):=G(tl^2;q)=\frac{1}{q}\sum_{r=1}^{q}e\Big(\frac{tl^2r^2}{q}\Big).\end{equation}
Set $b=\gcd(l,q)$ and  write $l=bL$, $q=bQ$ with $\gcd(L,Q)=1$.
Substituting this into \eqref{Fl} yields
\[
F(l)=\frac{1}{bQ}\sum_{r=1}^{bQ}e\Big(\frac{tb^2L^2r^2}{bQ}\Big)
=\frac{1}{Q}\sum_{r=1}^{Q}e\Big(\frac{tbL^2r^2}{Q}\Big).
\]
Since $\gcd(L,Q)=1$, as $r$ runs through a complete residue system modulo $Q$,
$Lr$ also runs through a complete residue system modulo $Q$. Hence
\[
F(l)=\frac{1}{Q}\sum_{r=1}^{Q}e\Big(\frac{tbr^2}{Q}\Big):=F(b).
\]
Combing this with \eqref{Wkx} gives
\[W_k(x)=\sum_{b|q}F(b)\sum_{\substack{l\le x\\ \gcd(l,q)=b}}\frac{\mu(l)}{l}.\]
Set $l=bL$, then $\mu(bL)\neq 0$ if and only if  $\gcd(L,b)=1$. Combining this with the above equation we obtain
\begin{align}
W_k(x)
&=\sum_{b|q}F(b)\sum_{\substack{L\le x/b\\ \gcd(l,q)=b\\ \gcd(L,b)=1}}\frac{\mu(bL)}{bL} \notag\\
&=\sum_{b|q}\frac{\mu(b)}{b}F(b)\sum_{\substack{L\le x/b\\ \gcd(L,q)=1}}\frac{\mu(L)}{L}\label{W_k}
\end{align}

\par By the Prime Number Theorem for arithmetic progressions, it is a standard result that
\[\lim_{X\to\infty}\sum_{j\le X, \gcd(j,q)=1}\frac{\mu(j)}{j}=0\]
for any fixed integer $q \ge 1$.
Combing this with \eqref{W_k} and $|F(l)|\le 1$ gives $\lim_{x\to\infty}W_k(x)=0$ uniformly for $k$, and so $|W_k(x)|\le C$ uniformly for $x$ and $k$.

\par From \eqref{TdN} we have
\[
T_{d,N}=\frac{1}{d^{1+\frac{\delta}{2}}}\sum_{k=1}^{N}\frac{h(k)}{k}W_k\Big(\Big\lfloor\frac{N}{k}\Big\rfloor\Big).
\]
Since $\sum_{k=1}^{\infty}\frac{|h(k)|}{k}<\infty$(see \cite{Wu-Zhan2026}), for any $\varepsilon>0$, we can choose $K>0$ such that
\[
\sum_{k>K}\frac{|h(k)|}{k}<\varepsilon.
\]
Then
\begin{align*}
|T_{d,N}|
&\le \sum_{k=1}^{K}\frac{|h(k)|}{k}\Big|W_k\Big(\Big\lfloor\frac{N}{k}\Big\rfloor\Big)\Big|
+C\sum_{k>K}\frac{|h(k)|}{k}\\
&\le \sum_{k=1}^{K}\frac{|h(k)|}{k}\Big|W_k\Big(\Big\lfloor\frac{N}{k}\Big\rfloor\Big)\Big|
+C\varepsilon.
\end{align*}
Let $N\to\infty$ and use $W_k(\lfloor N/k\rfloor)\to 0$ for each fixed $k\le K$. It follows that
\[
\limsup_{N\to\infty}|T_{d,N}|\le C\varepsilon.
\]
By the arbitrariness of $\varepsilon$, we obtain $S_d(a,q)=0$.}
\end{example}

\end{document}